\documentclass[12pt]{article}
\usepackage[T2A]{fontenc}
\usepackage[cp1251]{inputenc}
\usepackage[english]{babel}
\usepackage{amsmath,latexsym,amsthm,amsfonts,tipa}
\usepackage{amssymb,amscd,graphpap}
\newcommand\NN{{\mathbb N}}

\newcommand\w{{\omega}}
\newcommand\kk{{\varkappa}}

\newcommand\PP{{\mathcal P}}

\newtheorem{theorem}{Theorem}[section]
\newtheorem{Th}{Theorem}[section]
\newtheorem{Lm}{Lemma}[section]
\newtheorem{Qs}{Question}[section]
\theoremstyle{definition}
\newtheorem{Ex}{Example}[section]
\newtheorem{Rm}{Remark}[section]

\begin{document}

\title{Partitions of groups into large subsets}
\author{Igor~Protasov, Sergii~Slobodianiuk}
\date{}
\maketitle

\begin{abstract}
Let $G$ be a group and let $\kk$ be a cardinal.
A subset $A$ of $G$ is called left (right) $\kk$-large if there exists a subset $F$ of $G$ such that $|F|<\kk$ and $G=FA$ ($G=AF$).
We say that $A$ is $\kk$-large if $A$ is left and right $\kk$-large.
It is known that every infinite group $G$ can be partitioned into countably many $\aleph_0$-large subsets.
On the other hand, every amenable (in particular Abelian) group $G$ cannot be partitioned into $>\aleph_0$ $\aleph_0$-large subsets.

We prove that every infinite group $G$ of cardinality $\kk$ can be partitioned into $\kk$ left-$\aleph_1$-large 
subsets and every free group $F_\kk$ in the infinite alphabet $\kk$ can be partitioned into $\kk$ $4$-large subsets.
\

{\bf 2010 AMS Classification}: 03E05, 05A18, 20B07

\

\end{abstract}

\section{Introduction}
Let $G$ be a group, $\kk$ be a cardinal, $[G]^{<\kk}$ denotes the family of all subsets of $G$ of cardinality $<\kk$.
A subset $A$ of $G$ is called
\begin{itemize}
\item{} {\em left (right)} $\kk$-large if there exists $F\in[G]^{<\kk}$ such that $G=FA$ ($G=AF$);
\item{} {\em $\kk$-large} if $A$ is left and right $\kk$-large.
\end{itemize}

We note that $A$ is left $\kk$-large if and only if $A^{-1}$ is right $\kk$-large.
In the dynamical terminology \cite[3, p.101]{b3}, the left $\aleph_0$-large subsets are known under the name {\em syndetic subsets}.

In \cite{b2}, Bella and Malykhin asked whether every infinite group $G$ can be partitioned into two $\aleph_0$-large subsets.
This question was answered in \cite{b5} (see also \cite[Theorem 3.12]{b8}): 
$G$ can be partitioned into countably many $\aleph_0$-large subsets.

If $G$ is amenable (in particular, $G$ is Abelian) and $\mu$ is a left invariant Banach measure on $G$ then $\mu(A)>0$ for every left $\aleph_0$-large subset $A$ of $G$.
It follows that $G$ cannot be partitioned into $>\aleph_0$ left $\aleph_0$-large subsets.

On the other hand (see \cite[Theorem 2.4]{b7} or \cite[Theorem 12.11]{b8}), every infinite group $G$ can be partitioned into $\kk$ left $\kk$-large subsets for each infinite cardinal $\kk\le|G|$.
In this connection, we asked \cite[Question 12.6]{b8} whether an infinite Abelian group $G$ of cardinality $\aleph_2$ can be partitioned into $\aleph_2$ $\aleph_1$-large subsets.

In section~\ref{s2} of this paper, we prove that every infinite group $G$ of cardinality $\kk$ can be partitioned 
into $\kk$ left $\aleph_1$-large subsets.

In section~\ref{s3} we partition the free group $F_\kk$ in the infinite alphabet $\kk$ into $\kk$ left $3$-large 
subsets and into $\kk$ $4$-large subsets.

In section~\ref{s4} we consider two alternative examples of partitions of $G$-spaces into large subsets and conclude 
the paper with some comments and open problems in Section~\ref{s5}.

\section{Partitions and filtrations}\label{s2}
Let $G$ be an infinite group with the identity $e$, $\kk$ be an infinite cardinal.
A family $\{G_\alpha:\alpha<\kk\}$ of subgroups of $G$ is called a {\em filtration} if the following conditions hold
\begin{itemize}
\item[(1)] $G_0=\{e\}$ and $G=\bigcup_{\alpha<\kk}G_\alpha$;
\item[(2)] $G_\alpha\subset G_\beta$ for all $\alpha<\beta<\kk$;
\item[(3)] $\bigcup_{\alpha<\beta}G_\alpha=G_\beta$ for each limit ordinal $\beta<\kk$.
\end{itemize}
Clearly, a countable group $G$ admits a filtration if and only if $G$ is not finitely generated.
Every uncountable group $G$ of cardinality $\kk$ admits a filtration satisfying the additional condition $|G_\alpha|<\kk$ for each $\alpha<\kappa$.

Following \cite{b6}, for each $0<\alpha<\kappa$, we decompose $G_{\alpha+1}\setminus G_\alpha$ into right cosets by 
$G_\alpha$ and choose some system $X_\alpha$ of representatives so $G_{\alpha+1}\setminus G_\alpha=G_\alpha X_\alpha$.
We take an arbitrary element $g\in G\setminus\{e\}$ and choose the smallest subgroup $G_\alpha$ with $g\in G_\alpha$.
By $(3)$, $\alpha=\alpha_1+1$ for some ordinal $\alpha_1<\kk$. 
Hence, $g\in G_{\alpha+1}\setminus G_{\alpha_1}$ and there exist $g_1\in G_{\alpha_1}$ and $x_{\alpha_1}\in X_{\alpha_1}$ such that $g=g_1x_{\alpha_1}$.
If $g_1\neq e$, we choose the ordinal $\alpha_2$ and elements $g_2\in G_{\alpha_2+1}\setminus G_{\alpha_2}$ and $x_{\alpha_2}\in X_{\alpha_2}$ such that $g_1=g_2x_{\alpha_2}$.
Since the set of ordinals $\{\alpha:\alpha<\kk\}$ is well-ordered, after finite number $s(g)$ of steps, we get the representation
$$g=x_{\alpha_{s(g)}}x_{\alpha_{s(g)-1}}\dots x_{\alpha_2}x_{\alpha_1},\text{ }x_{\alpha_i}\in X_{\alpha_i}.$$
We note that this representation is unique and put
$$\gamma_1(g)=\alpha_1,\text{ }\gamma_2(g)=\alpha_2,\text{ }\dots,\gamma_{s(g)}(g)=\alpha_{s(g)}.$$
\begin{theorem}\label{t2.1}
Every infinite group $G$ of cardinality $\kk$ can be partitioned into $\kk$ left $\aleph_1$-large subsets.
\end{theorem}
\begin{proof}
If $G$ is countable, the statement is evident because each singleton is $\aleph_1$-large.
Assume that $\kk>\aleph_0$ and fix some filtration $\{G_\alpha:\alpha<\kk\}$ of $G$ such that $|G_1|=\aleph_0$.
Given any $g\in G\setminus\{e\}$, we rewrite the canonical representation $g=x_{\gamma_n}\dots x_{\gamma_1},$
in the following $$g=g_1x_{\gamma_m}\dots x_{\gamma_1},$$ $g_1\in G_1$, $0<\gamma_m<\dots<\gamma_1$. 
Here $g_1=e$ and $m=n$ if $\gamma_n>0$, and $g_1=x_{\gamma_n}$ and $m=n-1$ if $\gamma_n=0$.
We put $\Gamma(g)=\{\gamma_1,\dots,\gamma_m\}$ and fix an arbitrary bijection $\pi:G_1\to\NN$.

We define a family $\{A_\alpha:0<\alpha<\kk\}$ of subsets of $G$ by the following rule: 
$g\in A_\alpha$ if and only if $\alpha\in\Gamma(g)$ and $\gamma_{\pi(g_1)}=\alpha$. Since the subsets 
$\{A_\alpha:0<\alpha<\kk\}$ are pairwise disjoint, it suffices to show that each $A_\alpha$ is left $\aleph_1$-large.
We take $a_\alpha\in X_\alpha$, put $F_\alpha=\{e,a_\alpha\}G_1$ and proof that $G=F_\alpha A_\alpha$.

Let $g\in G$ and $\alpha\in \Gamma(g)$. By the definition of $A_\alpha$, there exists $h\in G_1$ such that 
$hg\in A_\alpha$ so $g\in G_1A_\alpha$. If $\alpha\notin\Gamma(g)$ then $\alpha\in\Gamma(a_\alpha^{-1}g)$ so
$a_\alpha^{-1}g\in G_1A_\alpha$ and $g\in a_\alpha G_1 A_\alpha$.
\end{proof}

\section{Partitions of free groups}\label{s3}
For a cardinal $\kk$, we denote by $F_\kk$ the free group in the alphabet $\kk$.
Given any  $g\in F_\kk\setminus\{e\}$ and $a\in\kk$, we write $\lambda(g)=a$ ($\rho(g)=a$) if the first (the last)
letter in the canonical representation of $g$ is either $a$ or $a^{-1}$.
\begin{Lm}\label{l3.1} Suppose that a group $G$ is a quotient of a group $H$, $f:H\to G$ is a quotient mapping. If $\PP$ is a partition of $G$ into $\kk$ $\lambda$-large subsets then $\{f^{-1}(P):P\in\PP\}$ is a partition of $H$ into $\kk$ $\lambda$-large subsets. \end{Lm}
\begin{proof} For each $g\in G$, we choose some element $h_g\in f^{-1}(g)$. If $G=XP$ or $G=PX$ then $H=\{h_x:x\in X\}f^{-1}(P)$ or $H=f^{-1}(P)\{h_x:x\in X\}$. \end{proof}
\begin{Th}\label{t3.1} 
For any infinite cardinal $\kk$, the following statements hold

$(i)$ $F_\kk$ can be partitioned into $\kk$ left $3$-large subsets;

$(ii)$ $F_\kk$ can be partitioned into $\kk$ $4$-large subsets.
\end{Th}
\begin{proof}
$(i)$ For each $a\in\kk$, we put $P_a=\{g\in F_\kk\setminus\{e\}:\lambda(g)=a\}$ and note that $F_\kk=\{e,a\}P_a$.

$(ii)$ We partition $\kk$ into $2$-element  subsets $\kk=\bigcup_{\alpha<\kk}\{x_\alpha,y_\alpha\}$ and put $X=\{x_\alpha:\alpha<\kk\}$, $Y=\{y_\alpha:\alpha<\kk\}$.

For every $\alpha<\kk$, we denote 
$$L_\alpha=\{g\in F_\kk\setminus\{e\}:\lambda(g)=x_\alpha\Leftrightarrow \rho(g)\in X, \lambda(g)=y_\alpha\Leftrightarrow \rho(g)\in Y,\}$$
$$R_\alpha=\{g\in F_\kk\setminus\{e\}:\rho(g)=x_\alpha\Leftrightarrow \lambda(g)\in Y, \rho(g)=y_\alpha\Leftrightarrow \rho(g)\in X,\}$$
Then we put $P_\alpha=L_\alpha\cup R_\alpha$ and note that the subsets $P_\alpha:\alpha<\kk$ are pairwise disjoint.

Given any $g\in F_\kk$, we have
$${\{e,x_\alpha,y_\alpha\}g\cap L_\alpha\neq\varnothing},\text{ }g\{e,x_\alpha,y_\alpha\}\cap R_\alpha\neq\varnothing.$$
Hence, $L_\alpha$ is left $4$-large and $R_\alpha$ is right $4$-large, so $P_\alpha$ is $4$-large.
\end{proof}
\begin{Rm}
Let $G$ be a group, $X$ be a left $3$-large subset of $G$, $Y$ be a right $3$-large subset of $G$ and $X\cap Y=\varnothing$.
We show that $G=X\cup Y$. In particular each group can be partitioned in at most $2$ $3$-large subsets.

Assume the contrary, take any $g\in G\setminus (X\cup Y)$.By the assumptions $G=\{e,x\}X=Y\{e,y\}$ for some $x,y\in G$.
Then $x^{-1}g\in X$, $gy^{-1}\in Y$ so $x^{-1}gy^{-1}\in X$ and $x^{-1}gy^{-1}\in Y$ contradicting $X\cap Y=\varnothing$.
\end{Rm}
\begin{Th}
For a natural number $n\ge2$, the following statements hold

$F_n$ can be partitioned into $\aleph_0$ left $3$-large subsets;

$F_n$ can be partitioned into $\aleph_0$ $5$-large subsets.
\end{Th}
\begin{proof}
Given any $g\in F_n\setminus\{e\}$, $a\in n$ and $m\in\NN$, we write $\overline{\lambda}(g)=a^m$ 
($\overline{\rho}(g)=a^m$) if $g=a^{\pm m}h$, $\lambda(h)\neq a$ ($g=ha^{\pm m}$, $\rho(h)\neq a$).

$(i)$ We fix two distinct letters $a,b$ from $n$ and, for each $m\in\NN$, put
$$P_m=\{g\in F_n\setminus\{e\}:\overline{\lambda}(g)=a^m\}$$
Clearly, $\{a^m,a^mb\}g\cap P_m\neq\varnothing$ for each $g\in F_n$, so $P_m$ is let $3$-large.

$(ii)$ We suppose that $n=2$ and $F_2$ is a free group in the alphabet $\{a,b\}$. For every $m\in\NN$ we denote
$$L_m=\{g\in F_2\setminus\{e\}:\overline{\lambda}(g)=a^m\Leftrightarrow\overline{\rho}(g)=a,\text{ }\overline{\lambda}(g)=b^m\Leftrightarrow\overline{\rho}(g)=b\},$$
$$R_m=\{g\in F_2\setminus\{e\}:\overline{\rho}(g)=a^m\Leftrightarrow\overline{\lambda}(g)=b,\text{ }\overline{\rho}(g)=b^m\Leftrightarrow\overline{\lambda}(g)=a\}.$$
Then we put $P_m=L_m\cup R_m$ and note that the subsets $\{P_m:m\in\NN\}$ are pairwise disjoint.

Given any $g\in F_2$, we have $$\{a^m,b^m,a^mb,b^ma\}g\cap L_m\neq\varnothing,\text{ }g\{a^m,b^m,ba^m,ab^m\}\cap R_m\neq\varnothing.$$
Hence, $L_m$ is left $5$-large, $R_m$ is right $5$-large, so $P_m$ is $5$-large.

If $n>2$ then $F_2$ is a quotient of $F_n$ and we can apply Lemma~\ref{l3.1}.
\end{proof}
\begin{Rm}
For every $\kk>1$, one can easily choose a disjoint family $\{x_nH_n:n\in\w\}$ of cosets of $F_\kk$ by some subgroup of finite index.
Let $G$ be an arbitrary group, $H$ be a subgroup of finite index, $a\in G$, $b\in G$. 
The set $aHb$ is called a {\em shifted subgroup} of finite index.
Clearly, each shifted subgroup of finite index is $\aleph_0$-large.
We show that $|\PP|\le\aleph_0$ for any family $\PP$ of pairwise disjoint shifted subgroups of $G$ of finite index.
We denote by $N$ the intersection of all normal subgroups of finite index of $G$, and let $f:G\to G/N$ denotes the quotient mapping.
Let $H$ be the profinite completion of $G/N$ and let $\mu$ be the Haar measure on $H$.
We observe that the subsets $\{cl_H f(P):P\in\PP\}$ are pairwise disjoint and $\mu(cl_H f(P))>0$ for each $P\in\PP$.
It follows that $|\PP|\le\aleph_0$.
\end{Rm}

\section{On partitions of $G$-spaces}\label{s4}
Let $G$ be a group and let $X$ be a transitive $G$-space with the action $G\times X\to X$, $(g,x)\mapsto gx$.
For a cardinal $\lambda$, a subset $Y$ of $X$ is called {\em $\lambda$-large} if there exists $F\in[G]^{<\lambda}$ such that $X=FY$.
\begin{Ex}
Let $\kk$ be an infinite cardinal, $X=\kk$ and let $G$ denotes the group of all permutations of $X$.
Let $Y$ be a subset of $X$ such that $|Y|=|X\setminus Y|$. We take a permutation $f\in G$ such that $f(Y)=X\setminus Y$, $f(X\setminus Y)=Y$.
Then $X=\{e,f\}Y$. It follows that $X$ can be partitioned into $\kappa$ $3$-large subsets.
\end{Ex}
\begin{Ex}
Let $\kk$ be an infinite cardinal, $X=\kk$ and let $G$ denotes the group of all permutations of $X$ with finite support.
Let $Y$ be a subset of $G$ such that $|X\setminus Y|=\kk$. We take an arbitrary $F\in[G]^{<\kk}$ and note that $(X\setminus Y)\setminus FY\neq\varnothing$.
It follows that $Y$ is not $\kk$-large. Hence, $X$ cannot be partitioned even into two $\kk$-large subsets.
\end{Ex}
\section{Comments and open questions}\label{s5}
$1$. We begin with two open question arising naturally from the results of the paper.
\begin{Qs}
Can every infinite group $G$ of cardinality $\kk$ be partitioned into $\kk$ $\aleph_1$-large subsets?
\end{Qs}
\begin{Qs}
For an infinite cardinal $\kk$, can the free group $F_\kk$ be partitioned into $\kk$ $3$-large subsets?
\end{Qs}
\begin{Qs}
For a natural number $n\ge2$, can the free group $F_n$ be partitioned into $\aleph_0$ $4$-large subsets?
\end{Qs}
$2.$ We do not know \cite[Problem 4.2]{b7} whether every infinite group $G$ of cardinality $\kk$ can be partitioned into $\kk$ $\kk$-large subsets.

$3.$ For every $n\in\NN$, there is a (minimal) natural number $\Phi(n)$ such that, for every group $G$ and every partition 
$G=A_1\cup\dots\cup A_n$ there exists a cell $A_i$ and a finite subset $F$ of $G$ such that $G=FA_iA_i^{-1}$ and $|F|\le\Phi(n)$.
So far it is an open problem \cite[Problem 13.44]{b4} whether $\Phi(n)=n$. For nowadays state of this problem see the survey \cite{b1}.

$4.$ In \cite{b9} we conjectured that every infinite group $G$ of cardinality $\kk$ can be partitioned 
$G=\bigcup_{n\in\w}A_n$  such that each subset $A_nA_n^{-1}$ is not left $\kk$-large.
We confirmed this conjecture for each group of regular cardinality and of some groups (in particular, Abelian) of an arbitrary cardinality.

$5.$ Every infinite group $G$ of regular cardinality $\kk$ can be partitioned $G=A_1\cup A_2$ such that $A_1$ and $A_2$ are not left $\kk$-large.
In \cite{b10} we show that this statement fails to be true for every Abelian group of singular cardinality $\kk$.

Department of Cybernetics,
\\
Kyiv University,
\\
Volodimirska 64,
\\
01033 Kyiv, Ukraine
\\
{\em Email address}: i.v.protasov@gmail.com
\\\\
Department of Mechanics and Mathematics,
\\
Kyiv University,
\\
Volodimirska 64,
\\
01033 Kyiv, Ukraine
\\
{\em Email address}: slobodianiuk@yandex.ru
\end{document}